\documentclass{article}
\usepackage{amsmath}
\usepackage{amssymb}

\newenvironment{proof}{\noindent {\bf Proof}}
{\hfill $\square$ \vspace{0.25cm}}

\newcommand{\ds}{\displaystyle}
\newcommand{\rr}{{\mathbb{R}}}
\newcommand{\nn}{{\mathbb{N}}}
\newcommand{\intot}{\ds\int_0^t}
\newcommand{\intotug}{\ds\int_0^t\int_0^\infty\int_G}
\newcommand{\intrd}{\ds\int_{\rr}}

\newcommand{\intg}{\ds\int_G}
\newcommand{\intgi}{\ds\int_{G_i}}
\newcommand{\sm}{{s-}}
\newcommand{\indiq}{{\bf 1}}
\newcommand{\ala}{\nonumber \\}
\newcommand{\cM}{{{\cal M}}}
\newcommand{\cH}{{{\cal H}}}
\newcommand{\cG}{{{\cal G}}}
\newcommand{\cB}{{{\cal B}}}
\newcommand{\cF}{{{\cal F}}}
\newcommand{\cL}{{{\cal L}}}
\newcommand{\tf}{{{\tilde f}}}
\newcommand{\crl}{{ \left( \begin{matrix} l \\ r \end{matrix} \right)}}
\newcommand{\cjl}{{ \left( \begin{matrix} l \\ j \end{matrix} \right)}}
\newcommand{\cnl}{{ \left( \begin{matrix} l \\ n \end{matrix} \right)}}

\newcommand{\e}{\epsilon}
\newcommand{\wbk}{{\bar{W}^{k,1}(\rr)}}
\newcommand{\vip}{\vskip0.2cm}

\newtheorem{theo}{Theorem}[section]
\newtheorem{pro}[theo]{Proposition}

\newtheorem{lem}[theo]{Lemma}

\begin{document}

\title{Smoothness of the law of some one-dimensional jumping S.D.E.s
with non-constant rate of jump}
\author{Nicolas {\sc Fournier}\footnote{ 
Centre de Maths, Facult\'e des Sciences et Technologies,
Universit\'e Paris 12, 61 avenue du G\'en\'eral de Gaulle, 
94010 Cr\'eteil Cedex, France. 
E-mail: {\tt nicolas.fournier@univ-paris12.fr}}}
\maketitle

\def\abstractname{Abstract}
\begin{abstract}
\noindent We consider a one-dimensional
jumping Markov process $\{X^x_t\}_{t \geq 0}$,
solving a Poisson-driven stochastic differential equation.
We prove that the law of $X^x_t$ admits a smooth density for $t>0$,
under some regularity and non-degeneracy assumptions on the coefficients 
of the S.D.E.
To our knowledge, our result is the first one including the 
important case of a non-constant rate of jump.
The main difficulty is that in such a case, the map $x \mapsto X^x_t$
is not smooth. This seems to make impossible the use of 
Malliavin calculus techniques.
To overcome this problem, we introduce a new method, in which 
the propagation of the smoothness of the density is obtained
by analytic arguments.
\end{abstract}

{\it Key words} : Stochastic differential equations, Jump processes, 
Regularity of the density.

{\it MSC 2000} : 60H10, 60J75.

\section{Introduction}\label{intro}

Consider a $\rr$-valued
Markov process with jumps $\{X^x_t\}_{t \geq 0}$, starting from 
$x \in \rr$, with generator $\cL$, defined for $\phi: \rr \mapsto \rr$
sufficiently smooth 
and $y \in \rr$, by
\begin{equation}\label{gi}
\cL \phi(y) = b(y) \phi' (y) + \gamma(y) \intg 
\left[ \phi(y+h(y,z)) -\phi(y) \right] q(dz),
\end{equation}
for some functions $\gamma,b: \rr \mapsto \rr$ with $\gamma$ nonnegative,
for some measurable space $G$ endowed with a nonnegative measure $q$,
and some function $h:\rr\times G \mapsto \rr$.

\vip

Roughly, $b(y)$ is the {\it drift} term: between $t$ and $t+dt$, 
$X^x_t$ moves from $y$ to $y+b(y)dt$. Next, $\gamma(y)q(dz)$ stands
for the {\it rate} at which $X^x_t$ jumps from
$y$ to $y+h(y,z)$.

\vip

We aim to investigate the smoothness of the law of
$X_t^x$ for $t>0$. 
Most of the known results are based on the use of
some Malliavin calculus, i.e. on a sort of
{\it differential calculus} with respect to the stochastic variable $\omega$.

\vip

The first results in this direction were obtained by Bismut \cite{b},
see also L\'eandre \cite{l}. 
Important results are due Bichteler et al. \cite{bgj}. 
We refer to Graham-M{\'e}l{\'e}ard \cite{gm}, Fournier \cite{fb2d}
and Fournier-Giet \cite{fg} for relevant applications to physic
integro-differential
equations such as the Boltzmann and the coagulation-fragmentation equations.
These results concern the case where $q(dz)$ is sufficiently smooth.

When $q$ is singular, Picard \cite{p} obtained some results
using some fine arguments relying on the affluence of small 
(possibly irregular) jumps. Denis \cite{d} and more recently Bally \cite{bally}
also obtained some regularity results when $q$ is singular, using 
the drift and the density of the jump instants, see also Nourdin-Simon
\cite{ns}.

\vip

All the previously cited works apply only to the case where 
the {\it rate of jump} $\gamma(y)$ is constant.
The case where
$\gamma$ is non constant is much more delicate. The main reason for this 
is that in such a case, the map $x \mapsto X^x_t$ cannot be regular 
(and even continuous). Indeed, if $\gamma(x)< \gamma(y)$, and if 
$q(G)= \infty$, then it is clear that for all small $t>0$,
$X^y$ jumps infinitely more often than $X^x$  before $t$. The only
available results with $\gamma$ not constant seem to be those of 
\cite{fou,fg}, where only the existence of a density was proved.
Bally \cite{bally} considers the case where $\gamma(y)q(dz)$
is replaced by something like $\gamma(y,z)q(dz)$, with
$\sup_y |\gamma(y,z)-1| \in L^1(q)$: the rate of jump is not constant,
but this concerns only finitely many jumps.

\vip

From a physical point of view, the situation where $\gamma$ is constant
is quite particular. For example in the (nonlinear)
Boltzmann equation, which describes
the distribution of velocities in a gas,
the rate of collision between two particles heavily depends on
their relative velocity (except in the so-called Maxwellian case treated
in \cite{gm,fb2d}).
In a fragmentation equation, describing the distribution of
masses in a system of particles subjected to breakage,
the rate at which a particle splits into smaller ones
will clearly almost always depend on its mass...

\vip

We will show here that when $q$ is smooth enough,
it is possible to obtain some regularity results
in the spirit of \cite{bgj}. 
Compared to \cite{bgj}, our result is

$\bullet$ stronger, since we allow $\gamma$ to be non-constant;

$\bullet$ weaker, since we are not able, at the moment, to study the case
of processes with infinite variations, and 
since we treat only the one-dimensional case
(our method could also apply to multidimensional processes, but
our non-degeneracy conditions would be very strong).

\vip

Our method relies on the following simple ideas:

(a) we consider, for $n\geq 1$, the first jump instant $\tau_n$ 
of the Poisson measure driving $X^x$, such that the corresponding {\it mark}
$Z_n$ falls in a
subset $G_n\subset G$ with $q(G_n)\simeq n$;

(b) using some smoothness assumptions on $q$ and $h$, we deduce that
$X^x_{\tau_n}$ has a smooth density (less and less smooth as $n$ tends to 
infinity);

(c) we also show that smoothness propagates with time in some sense,
so that $X^x_t$ has a smooth density conditionnally to $\{t\geq \tau_n\}$;

(d) we conclude by choosing carefully $n$ very large in such a way that 
$\{t \geq \tau_n\}$ occurs with sufficiently great probability.

\vip

As a conclusion, we obtain the smoothness of the density using only
the regularizing property of {\it one} (well-chosen) jump.
On the contrary, Bichteler et al. \cite{bgj} were
using the regularization of infinitely many jumps, which
was possible using a sort of Malliavin calculus.
Surprisingly, our non-degeneracy condition does not seem
to be stronger, see Subsection \ref{tictic} 
for a detailed comparison in a particular (but quite typical) example.

%\vip
%To conclude this introduction, let us mention that Bichteler-Jacod
%\cite[Example 1, page 135]{bj} proposed an alternative way to study 
%our problem: one may 
%find, under some conditions, a function $H$ such that (\ref{gi}) rewrites
%\begin{equation*}
%\cL \phi(y) = b(y) \phi' (y) + \intg 
%\left[ \phi(y+H(y,z)) -\phi(y) \right] q(dz).
%\end{equation*}
%Unfortunately, the function $H$ is defined through $h$ and $\gamma$ in
%quite a complicated way. Thus conditions on $H$ cannot be clearly 
%written in terms of assumptions concerning $h$ and $\gamma$,
%and this leads to quite restrictive conditions, see \cite{bj}.

\vip

We present our results in Section \ref{results},
and we give the proofs in Sections \ref{mainproof} and \ref{propasmooth}.
An Appendix lies at the end of the paper.

\section{Results}\label{results}

In the whole paper, $\nn=\{1,2,...\}$.
Consider the one-dimensional S.D.E.
\begin{equation}\label{sde}
X_t^x = x + \intot b(X_s^x) ds+ \intotug 
h(X_\sm^x,z)\indiq_{\{u\leq \gamma(X_\sm^x)\}}N(ds,du,dz),
\end{equation}
where

\vip

{\bf Assumption $(I)$:} The Poisson measure $N(ds,du,dz)$
on $[0,\infty) \times [0,\infty)\times G$ has the intensity
measure $ds du q(dz)$, for some
measurable space $(G,\cG)$ endowed with a nonnegative measure $q$.
For each $t\geq 0$ we set $\cF_t:=\sigma\{N(A), A \in \cB([0,t])\otimes
\cB([0,\infty))\otimes\cG\}$. 

\vip

We will require some smoothness of the coefficients.
For $f(y):\rr \mapsto \rr$ (and $h(y,z):\rr\times G \mapsto \rr$), 
we will denote by $f^{(l)}$ (and $h^{(l)}$) the $l$-th
derivative of $f$ (resp. of $h$ with respect to $y$). Below,
$k \in \nn$ and $p\in [1,\infty)$ are fixed.

\vip

{\bf Assumption $(A_{k,p})$:} The functions
$b:\rr \mapsto\rr $ and $\gamma:\rr \mapsto\rr_+ $
are of class $C^k$, with all their derivatives
of order $0$ to $k$ bounded.

The function $h:\rr \times G \mapsto\rr$ is measurable, and for each
$z \in G$, $y\mapsto h(y,z)$ is of class $C^k$ on $\rr$. 
There exists $\eta \in (L^1\cap L^p)(G, q)$ 
such that for all $y\in \rr$, all $z\in G$, all $l\in \{0,...,k\}$,
$|h^{(l)} (y,z)| \leq \eta(z)$.

\vip

Under $(A_{1,1})$, 
$\cL \phi$, introduced in (\ref{gi}), is well-defined for all 
$\phi \in C^1(\rr)$ with a bounded derivative. 
The following result classically holds, see e.g. \cite[Section 2]{fou}
for the proof of a similar statement.

\begin{pro}\label{exist}
Assume $(I)$ and $(A_{k,p})$ for some $p\geq 1$, some $k\geq 1$.
For any $x\in \rr$, there exists a unique c{\`a}dl{\`a}g 
$({\cal F}_t )_{t\geq 0}$-adapted process $(X^x_t)_{t\geq 0}$ solution to
(\ref{sde}) such that for all all $T\in [0,\infty)$, 
$E [ \sup_{s\in [0,T]} |X_s^x|^p ] <\infty$.

The process  $(X^x_t)_{t\geq 0,x\in \rr}$  is a strong Markov process
with generator $\cL$ defined by (\ref{gi}). We will denote by
$p(t,x,dy):=\cL(X^x_t)$ its semi-group.
\end{pro}

\subsection{Propagation of smoothness}

We consider the space $\cM(\rr)$ of finite (signed) measures on $\rr$,
and we abusively write $||f||_{L^1(\rr)}:=||f||_{TV}=\int_{\rr} |f|(dy)$ 
for $f \in \cM(\rr)$. We denote by $C^k_b(\rr)$ (resp. $C^k_c(\rr)$)
the set of $C^k$-functions with all their derivatives bounded 
(resp. compactly supported). 
We introduce, for $k \geq 1$, the space $\wbk$ of measures $f\in \cM(\rr)$
such that for all $l \in \{1,...,k\}$, 
there exists $g_l \in \cM(\rr)$ such that
for all $\phi \in C^k_c(\rr)$ (and thus for all  $\phi \in C^k_b(\rr)$),
\begin{equation*}
\intrd f(dy) \phi^{(l)}(y) =
(-1)^{l} \intrd g_l(dy)  \phi(y) .
\end{equation*}
If so, we set $f^{(l)} = g_l$. Classically, 
for $f \in \cM(\rr)$, $f \in \wbk$ if and only if
\begin{equation}\label{dfn}
||f||_{\wbk}
:= \sum_{l=0}^k 
\sup \left\{ \intrd f(dy)\phi^{(l)}(y), \;\;
\phi \in C^k_b(\rr),\; ||\phi||_\infty \leq 1 \right\}
\end{equation}
is finite (here $C^k_b$ could be replaced by $C^k_c$, $C^\infty_b$,
or $C^\infty_c$), and in such a case,
\begin{eqnarray*}
||f||_{\wbk} = \sum_{l=0}^k ||f^{(l)} ||_{L^1(\rr)}.
\end{eqnarray*}
Let us finally recall that

$\bullet$ for $f\in C^k(\rr)$,
$f(y)dy$ belongs to $\wbk$ if and only if $\sum_0^k |f^{(l)}| \in L^1(\rr)$;

$\bullet$ if $f \in \wbk$, with $k\geq 2$, then $f(dy)$ has a density
of class $C^{k-2}(\rr)$.

\vip

We now introduce a first non-degeneracy assumption (here $h'(y,z)=
\partial_y h(y,z)$).

\vip

{\bf Assumption $(S)$:}
There exists $c_0>0$ such that for all $z\in G$, all $y \in \rr$, 
$1+h'(y,z)\geq c_0$.

\begin{pro}\label{propa}
Let $p\geq k+1 \geq 2$ be fixed, assume $(I)$, $(A_{k+1,p})$, and $(S)$.
For $t\geq 0$ and a probability measure $f$ on $\rr$, we define
$p(t,f,dy)$ on $\rr$ by $p(t,f,A)= \int_{\rr} f(dx)p(t,x,A)$,
where $p(t,x,dy)$ was defined in Proposition \ref{exist}.

There is $C_k>0$ such that for
all probability measures $f\in\wbk$, all $t\geq 0$,
\begin{equation*}
||p(t,f,.)||_{\wbk} \leq ||f||_{\wbk}e^{C_k t}.
\end{equation*}
\end{pro}

\vip

Observe that $p(t,f,dy)$ is the law of $X^{X_0}_t$ where
$(X^x_t)_{t\geq 0,x\in\rr}$ solves (\ref{sde}) and 
where $X_0\sim f(dy)$ is independent of $N$.

\vip

Assumption $(S)$ is probably far from optimal, 
but something in this spirit is needed: take
$b\equiv 0$, $\gamma \equiv 1$ and 
$h(y,z)=-y\indiq_{A}(z) +y \eta(z)$ for some $A\subset G$ with $q(A)<\infty$
and some $\eta \in L^1(G,q)$. Of course, $(S)$ is not satisfied, and
one easily checks that there exists
$\tau_A$ exponentially distributed (with parameter $q(A)$) 
such that a.s., for all 
$t\geq \tau_A$, all $x\in \rr$, $X^x_t=0$. This forbids
the propagation of smoothness, since then $p(t,f,dy) \geq (1-e^{-q(A)t})
\delta_0(dy)$,
even if $f$ is smooth.

\subsection{Regularization}

We now give the non-degeneracy condition
that will provide a smooth density to our process.
A generic example of application (in the spirit of \cite{bgj}) 
will be given below.
For two nonnegative measures $\nu,\tilde\nu$ on $G$, we say
that $\nu \leq \tilde\nu$ if for all $A\in\cG$, $\nu(A)\leq \tilde\nu(A)$.
Here $k\in \nn$, $p\in [0,\infty)$ and $\theta>0$.

\vip

{\bf Assumption $(H_{k,p,\theta})$:}
Consider the jump kernel $\mu(y,du)$ associated to our process, defined
by $\mu(y,A)=\gamma(y) \int_G \indiq_A(h(y,z)) q(dz)$ 
(which may be infinite) for all $A\in \cB(\rr)$.

There exists a (measurable) family $(\mu_n(y,du))_{n\geq 1, y \in \rr}$  
of measures on $\rr$ meeting the following points:

(i) for $n\geq 1$, $y\in G$, $0 \leq \mu_n(y,du) \leq \mu(y,du)$ 
and $\mu_n(y,\rr)\geq n $;

(ii) for all $r>0$, $n\geq 1$, $\sup_{|y|\leq r} \mu_n(y,\rr) <\infty $;

(iii) there exists $C>0$ such that for all $n\in\nn$, $y\in\rr$,
\begin{equation*}
\frac{1}{\mu_n(y,\rr)}|| \mu_n(y,.)||_{ \wbk} \leq C(1+|y|^p)  e^{\theta n}.
\end{equation*}

The principle of this assumption is quite natural: it says that at
any position $y$, our process will have sufficiently many jumps with 
a sufficiently smooth density.
Our main result is the following.

\begin{theo}\label{result}
Let $p\geq {k+1} \geq 3$ and $\theta>0$ be fixed. 
Assume $(I)$, $(A_{k+1,p})$, $(S)$ and $(H_{k,p,\theta})$.
Consider the law $p(t,x,dy)$ at time $t\geq 0$ 
of the solution $(X^x_t)_{t\geq 0}$ to (\ref{sde}).

(a) Let $t>\theta /(k-1)$. For any $x\in \rr$,
$p(t,x,dy)$ has a density $y\mapsto p(t,x,y)$ 
of class $C^n_b(\rr)$ as soon as $0 \leq n< kt/(\theta+t)-1$.

(b) In particular, if $(H_{k,p,\theta})$ holds for all $\theta>0$, then
for all $t>0$, all $x\in\rr$, 
$y \mapsto p(t,x,y)$ is of class $C^{k-2}_b(\rr)$.
\end{theo}

\subsection{Another assumption}\label{toctoc}

It might seem strange to state our regularity assumptions
with the help of $\gamma,h,q$, and to our nondegeneracy conditions
with the help of the jump kernel $\mu$. However, it seems to us
to be the best way to give understandable assumptions.

\vip

Let us give some conditions on $\gamma$, $h$, $q$, 
in the spirit of \cite{bgj}, which imply $(H_{k,p,\theta})$. 

\vip

{\bf Assumption $(B_{k,p,\theta})$:}
$G=\rr$, and 
for all $y\in\rr$, $\gamma(y)>0$ and there exists 
$I(y)=(a(y),\infty)$ (or $(-\infty, a(y))$) with $a(y)\in \rr$,
with $y \mapsto a(y)$ measurable,
such that $q(dz) \geq \indiq_{I(y)}(z)dz$ and such that the following 
conditions are fulfilled:

(a) for all $y \in \rr$, $z \mapsto h(y,z)$ is of class $C^{k+1}$ on 
$I(y)$. The derivatives $h^{(l)}_z$ (w.r.t. $z$)
for $l=1,...,k+1$ are uniformly bounded on $\{(y,z);\;
y \in \rr, z\in I(y)\}$;

(b) for all $y\in\rr$, all $z\in I(y)$, $h'_z(y,z) \ne 0$, and 
with $I_n(y)=[a(y),a(y)+n/\gamma(y)]$ (or $[a(y)-n/\gamma(y),a(y)] $),
\begin{equation}\label{albr}
\frac{\gamma(y)}{n}\int_{I_n(y)}
|h'_z(y,z)|^{-2k} dz \leq C(1+|y|^p) e^{\theta n}.
\end{equation}
 
\begin{lem} \label{tract}
$(B_{k,p,\theta})$ and $(A_{1,1})$ imply $(H_{k,p,\theta})$.
\end{lem}

This lemma is proved in the Appendix.
Let us give some examples for (\ref{albr}).

\vip

{\bf Examples:} Assume
that $|h'_z(y,z)| \geq \e (1+|y|)^{-\alpha} \zeta(z)$, for all $y\in\rr$,
all $z\in I(y)$, 
for some $\alpha\geq 0$, $\e>0$.

\vip

$\bullet$ If $\zeta(z)= (1+|z|)^{-\delta}$, for some
$\delta\geq 0$, and $\gamma(y) \geq c (1+|y|)^{-\beta}$ 
for some $c>0$, $\beta \geq 0$,
then (\ref{albr}) holds for all $k\geq 1$, all 
$\theta>0$ and all $p \geq 2k (\alpha + \beta\delta)$.

\vip

$\bullet$ If $\zeta(z)= e^{-d |z|^\delta}$, for some $d>0$,
$\delta \in (0,1)$, and if $\gamma(y) \geq c [ \log (2+|y|)]^{-\beta}$,
with $c>0$, $\beta \in [0,(1-\delta)/\delta)$, 
then (\ref{albr}) holds for all $k\geq 1$, all $\theta>0$, 
all $p>2 k \alpha$.

\vip

$\bullet$ If $\zeta(z)= e^{-d |z|}$, for some $d>0$,
if $\gamma(y)\geq c>0$,
then (\ref{albr}) holds  for all $k \geq 1$, all $\theta\geq 2kd/c$ and 
all $p\geq 2 k \alpha$.

\vip

$\bullet$ With our assumption that $\gamma$ is bounded, 
(\ref{albr}) does never hold if $\zeta(z)= e^{-d |z|^\delta}$
for some $d>0$, $\delta>1$.

\vip

Observe on these examples that there is a balance between the 
{\it rate} of jump $\gamma$
and the {\it regularization power} of jumps (given,
in some sense, by lowerbounds of $|h'_z|$).  The more the power
of regularization is small, the more the rate of jump has to be bounded
from below.
This is quite natural and satisfying.

\subsection{Comments}\label{tictic}

Let us mention that when $\gamma$ is constant,
the result of \cite{bgj} (in dimension $1$), is essentially the following.
Roughly, they also assume something like $q(dz) \geq \indiq_{(a,\infty)}(z)dz$
(they actually consider the case where $q(dz)\geq \indiq_O(z) dz$ for
some infinite open subset $O$ of $\rr$).

They assume more integrability on the coefficients 
(something like $(A_{k,p})$ for all $p>1$).
They assume $(S)$, and much more joint
regularity (in $y,z$)  of $h$ (see Assumption $(A-r)$ page 9 in \cite{bgj}),
the uniform boundedness of $\partial_{z^\alpha} \partial_{y^\beta} h$
as soon as $\alpha\geq 1$.

Their non-degeneracy condition (see Assumption $(SB-(\zeta,\theta))$ 
page 14
in \cite{bgj}) is of the form
$|h'_z(y,z)|^2 \geq  \e (1+|x|)^{-\delta}\zeta(z)$, for some $\delta \geq 0$,
some $\e>0$, 
and some {\it broad} function $\zeta$ (see Definition 2-20 
and example 2-35 pages 13 and 17 in \cite{bgj}). This
notion is probably not exactly comparable to (\ref{albr}). 
Roughly, 

$\bullet$ when $\zeta(z)= e^{-\alpha |z|^\delta}$ with $\delta>1$,
their result does not apply (as ours);

$\bullet$ when  $\zeta(z)= e^{-\alpha |z|^\delta}$ with $\delta<1$, 
or when  $\zeta(z)= (1+|z|)^{-\beta}$ with $\beta>0$, their
result applies for all times $t>0$ (as ours);

$\bullet$ when $\zeta(z)= e^{-\alpha |z|}$, their result applies
for sufficiently large times (as ours).

\vip

As a conclusion, we have slightly less technical assumptions. 
About the nondegeneracy assumption, it seems that
the condition in \cite{bgj} and ours are very similar (when $\gamma \equiv 1$).
Let us insist on the fact that this is quite surprising: one could
think that since we use only the regularization of {\it one} jump,
our nondegeneracy condition should be much stronger than that of \cite{bgj}.

\vip

We could probably state an assumption as
$(B_{k,p,\theta})$ for a general lowerbound of the form
$q(dz) \geq \indiq_O(z)\varphi(z)dz$, for some
open subset $O$ of $\rr$ and some $C^\infty$ function $\varphi: O \mapsto \rr$,
but this would be very technical.

\vip

Finally, it seems highly probable that one may assume,
instead of $(S)$, that $0< 1/(1+h'(x,z)) \leq \alpha(z) \in L^1 \cap L^r(G,q)$
(with $r$ large enough); and that the assumptions $b$, $\gamma$ bounded
and $|h(x,z)| \leq \eta(z)$ (in $(A_{k,p})$) could 
be replaced by
$|b(x)| \leq C(1+|x|)$ and $\gamma(x) |h(x,z)|\leq (1+|x|)\eta(z)$,
with $\eta \in L^1 \cap L^p(G,q)$.
However, the paper is technical enough.

%{\bf Assumption $(C_{k,p,\theta})$:} 
%For all $y\in\rr$, $\gamma(y)>0$, $G=\rr$,
%$q(dz) \geq \varphi(z)\indiq_O(z)dz$, for some open subset $O$ of
%$\rr$, some nonnegative function $\varphi$ of class $C^\infty$ on $O$.
%
%For each $y \in \rr$,
%$z \mapsto h(y,z)$ is of class $C^{k+1}$ on $O$.
%
%The derivatives (w.r.t. $z$) $h^{(l)}_z$ 
%for $l=1,...,k$ are bounded on $\rr\times O$.
%
%There exists a function $\rho:\rr \mapsto [0,1]$, of class $C^\infty_b$
%on $\rr$, supported by $O$, such that ???????????????????
%
%and ???????????????? 
%\begin{equation}\label{albr2}
%\frac{\gamma(y)}{n}\int_a^{a+n/\gamma(y)}
%\rho(z)^{-1} |h'_z(y,z)|^{-2k+1} \varphi(z) dz \leq C(1+|y|^p) e^{\theta n}.
%\end{equation}
%
%Of course this assumption is not very transparent,
%because of the function $\rho$, but this seems
%to be unavoidable, see 2-23 and 2-24 page 14 in \cite{bgj}.

\vip

We prove Theorem \ref{result} in Section \ref{mainproof}
and Proposition \ref{propa} in Section \ref{propasmooth}.

\section{Smoothness of the density}\label{mainproof}

In this section, we assume that Proposition \ref{propa} holds,
and we give the proof of our main result.
We refer to the introduction for the main ideas of the proof.

\vip

\begin{proof} {\bf of Theorem \ref{result}.} 

We consider here $x \in \rr$, the associated process $(X^x_t)_{t\geq 0}$.
We assume $(I)$, $(S)$, $(A_{k+1,p})$, and $(H_{k,p,\theta})$ 
for some $p\geq k+1 \geq 3$, some $\theta>0$.
Due to Proposition \ref{exist},
\begin{equation}\label{mp}
\forall \; t>0, \quad C_t := E\left[\sup_{[0,t]} |X^x_s|^p\right]<\infty.
\end{equation}
Recall $(H_{k,p,\theta})$, and denote by $f_n(y,u)$ the density (bounded
by $1$) of $\mu_n(y,du)$ with respect to $\mu(y,du)$. Then 
for $q_n(y,dz):=d_n(y,z)q(dz)$ with
$d_n(y,z):=f_n(y,h(y,z))$ (which is bounded by $1$), one easily checks
that for
all $A\in \cB(\rr)$,
$\mu_n(y,A)=\gamma(y) \int_G \indiq_A(h(y,z))q_n(y,dz)$.
As a consequence, still using $(H_{k,p,\theta})$, 

(i) $0\leq q_n(y,dz) \leq q(dz)$, and $\gamma(y) q_n(y,G)=\mu_n(y,\rr) 
\geq n$;

(ii) for all $r>0$, $n\in\nn$, $\sup_{|y|\leq r} \gamma(y) q_n(y,G)<\infty$.

We now divide the proof into four parts.

\vip

{\bf Step 1.} We first introduce some well-chosen instants of jump
that will provide a density to our process. To this end, 
we write $N=\sum_{i\geq 1} \delta_{(t_i,u_i,z_i)}$, we consider 
a family of i.i.d. random variables $(v_i)_{i\geq 1}$ uniformly distributed
on $[0,1]$, independent of $N$. 
We introduce the Poisson measure 
$M=\sum_{i\geq 1} \delta_{(t_i,u_i,z_i,v_i)}$ on $[0,\infty)\times[0,\infty)
\times G \times [0,1]$ with intensity measure $ds du q(dz)dv$.
Then we observe that $N(ds,du,dz)=M(ds,du,dz,[0,1])$.
Let $\cH_t=\sigma\{M(A),\; A\in \cB([0,t])\otimes
\cB([0,\infty))\otimes \cG \otimes 
\cB([0,1])\}$.

Next, we observe, using point (ii) above and (\ref{mp}), that a.s.,
for all $t\geq 0$,
\begin{eqnarray*}
&\sup_{[0,t]} \int_0^\infty \int_G \int_0^1
\indiq_{\{u\leq \gamma(X^x_\sm), v \leq d_n(X^x_\sm,z)\}} du q(dz) dv \\
&= \sup_{[0,t]} \gamma(X^x_\sm)q_n(X^x_\sm,G)<\infty.
\end{eqnarray*}
We thus may consider, for each $n\geq 1$, the a.s. positive
$(\cH_t)_{t\geq 0}$-stopping time
\begin{equation*}
\tau_n=\inf \left\{ t \geq 0; \int_0^t \int_0^\infty \int_G \int_0^1
\indiq_{\{u\leq \gamma(X^x_\sm), v \leq d_n(X^x_\sm,z)\}} M(ds,du,dz,dv)>0
 \right\},
\end{equation*}
and the associated {\it mark} $(U_n,Z_n,V_n)$ of $M$.
Then one easily checks that

(a) for $t\geq 0$, $P[\tau_n \geq t] \leq e^{-n t}$, since 
due to point (i), a.s., for all $s\geq 0$,
\begin{eqnarray*}
\int_0^\infty \int_G \int_0^1
\indiq_{\{u\leq \gamma(X^x_\sm), v \leq d_n(X^x_\sm,z)\}} du q(dz) dv
&=&\gamma(X^x_\sm) \int_G d_n(X^x_\sm)q(dz) \ala
&=&\gamma(X^x_\sm)q_n(X^x_\sm,G)\geq n;
\end{eqnarray*}

(b) $U_n \leq \gamma(X^x_{\tau_n-})$ a.s. by construction;

(c) conditionnally to $\cH_{\tau_n-}$, $Z_n \sim q_n(X^x_{\tau_n-},dz)/q_n(X^x_{\tau_n-},G)$. 
Indeed, the triple $(U_n,Z_n,V_n)$ classically follows, conditionnally to $\cH_{\tau_n-}$, the distribution
\begin{equation*}
\frac{1}{\gamma(X^x_{\tau_n-})q_n(X^x_{\tau_n-},G)}
\indiq_{\{u\leq \gamma(X^x_{\tau_n-}), v \leq d_n(X^x_{\tau_n-},z)\}} du q(dz) dv,
\end{equation*}
and it then suffices to integrate over
$u\in [0,\infty)$ and $v \in [0,1]$ and to use that $d_n(y,z)q(dz)=q_n(y,dz)$.

\vip

{\bf Step 2.} By construction and due to Step 1-(b),
\begin{equation*}
X^x_{\tau_n} = X^x_{\tau_n-} + h( X^x_{\tau_n-},Z_n)
\indiq_{\{U_n\leq \gamma(X^x_{\tau_n-})\}}=X^x_{\tau_n-} + h( X^x_{\tau_n-},Z_n).
\end{equation*}
Hence conditionnally to $\cH_{\tau_n-}$, the law of 
$X^x_{\tau_n}$ is $g_n(\omega,dy) :=
\mu_n(X^x_{\tau_n-},dy-X^x_{\tau_n-})/\mu_n(X^x_{\tau_n-},\rr)$. Indeed,
for any  bounded measurable function $\phi:\rr \mapsto \rr$,
using Step 1-(c) and that $\mu_n(y,A)=\gamma(y)\int_G \indiq_A(h(y,z))
q_n(y,dz)$,
\begin{eqnarray*}
E\left[ \phi(X^x_{\tau_n}) \vert \cH_{\tau_n-}\right]=
\int_G \phi[ X^x_{\tau_n-} + h( X^x_{\tau_n-},z)]
\frac{q_n(X^x_{\tau_n-},dz)}{q_n(X^x_{\tau_n-},G)}
\hskip1.5cm\ala
= \int_\rr \phi( X^x_{\tau_n-} + y) \frac{\mu_n(X^x_{\tau_n-} ,dy)}
{\mu_n(X^x_{\tau_n-},\rr)} 
= \int_\rr \phi(y) g_n(dy) .
\end{eqnarray*}
Due to assumption $(H_{k,p,\theta})$, we know that for some constant
$C$, a.s.,
\begin{equation}\label{gw}
||g_n||_{\wbk}=\frac{1}{\mu_n(X^x_{\tau_n-},\rr)}||\mu_n(X^x_{\tau_n-},.)||_{\wbk} 
\leq C(1+|X^x_{\tau_n-}|^p) e^{\theta n}.
\end{equation}

{\bf Step 3.} We now use the strong Markov property. For $t\geq 0$
and $n \geq 1$, for $\phi: \rr \mapsto \rr$, with the
notation of Proposition \ref{propa}, since 
$\{t\geq \tau_n\} \in \cH_{\tau_n-}$,
\begin{eqnarray}\label{ippaf}
E[\phi(X^x_t)]=E[\phi(X^x_t) \indiq_{\{t<\tau_n\}}]
+ E \left[\indiq_{\{t\geq \tau_n\}} \intrd \phi(y)p(t-\tau_n,g_n,dy)\right].
\end{eqnarray}
But from Proposition \ref{propa} and (\ref{gw}), there exists a
constant $C_{t,k}$ such that a.s.
\begin{eqnarray}\label{gw2}
\indiq_{\{t\geq \tau_n\}} ||p(t-\tau_n,g_n,.)||_{\wbk} &\leq& 
C_{t,k}\indiq_{\{t\geq \tau_n\}} \sup_{[0,t]}(1+|X^x_{s}|^p) 
e^{\theta n}.
\end{eqnarray}

{\bf Step 4.} Consider finally the application $\psi(\xi,y)=e^{i \xi y}$.
Then the Fourier transform of 
the law $p(t,x,dy)$ of $X^x_t$ is given by 
$\hat p_{t,x}(\xi):=E[\psi(\xi,X^x_t)]$.
We apply (\ref{ippaf}) with the choice
$\phi(y)= \psi^{(k)}(\xi,y)= (i\xi)^k\psi(\xi,y)$.
We get, for $n\geq 1$, $\xi \in \rr$,
\begin{eqnarray}\label{ets}
|\xi|^k |\hat p_{t,x}(\xi)| \leq |\xi|^k P[\tau_n>t]
+ E \left[\indiq_{\{t\geq \tau_n\}} \left|
\intrd \psi^{(k)}(\xi,y)p(t-\tau_n,g_n,dy)\right|\right].
\end{eqnarray}
But on $\{t \geq \tau_n\}$, an integration by parts and
then (\ref{gw2}) leads us to
\begin{eqnarray*}
\left|\intrd \psi^{(k)}(\xi,y)p(t-\tau_n,g_n,dy)\right|
=\left|\intrd \psi(\xi,y)  p^{(k)}(t-\tau_n,g_n,dy)\right|\ala
\leq || \psi(\xi,.) ||_\infty ||p(t-\tau_n,g_n,.)||_{\wbk}
\leq C_{t,k} e^{\theta n}\sup_{[0,t]}(1+|X^x_{s}|^p) .
\end{eqnarray*}
Hence (\ref{ets}) becomes, using Step 1-(a) and
(\ref{mp}),
\begin{equation*}
|\xi|^k |\hat p_{t,x}(\xi)| \leq |\xi|^k e^{-nt} + C_{t,k}(1+C_t) 
e^{\theta n}.
\end{equation*}
Choosing for $n$ the integer part of $\frac{k}{\theta+t} \log |\xi|$,
we obtain, for some constant $A_t$,
\begin{equation*}
|\xi|^k |\hat p_{t,x}(\xi)| \leq (e^t + C_{t,k}(1+C_t)) 
|\xi|^{k \theta/(\theta+t)}=: A_t |\xi|^{k \theta/(\theta+t)}.
\end{equation*}
Since on the other hand  $|\hat p_{t,x}(\xi)|$ is clearly bounded
by $1$, we deduce that
\begin{equation}\label{ets4}
|\hat p_{t,x}(\xi)| \leq 1 \land A_t |\xi|^{- k t/(\theta+t)}.
\end{equation}
Let finally $n\geq 0$ such that $n< \frac{k t}{\theta+t}-1$, which is possible
if $t> \frac{\theta }{k-1}$. Then (\ref{ets4}) ensures us that
$|\xi|^n |\hat p_{t,x}(\xi)|$ belongs to $L^1(\rr,d\xi)$, which classically
implies that $p(t,x,dy)$ has a density of class $C^n_b(\rr)$.
\end{proof}

\section{Propagation of smoothness}\label{propasmooth}

It remains to prove Proposition \ref{propa}. It is very technical,
but the principle is quite simple: we study the 
Fokker-Planck integro-partial-differential equation 
associated with our process,
and show that if the initial condition is smooth, so is the solution
for all times, in the sense of $\wbk$ spaces.

\vip

In the whole section, $K$ is a constant whose value may change from line to line,
and which depends only on $k$ and on the bounds of the coefficients assumed
in assumptions $(A_{k+1,p})$ and $(S)$.

\vip

For functions $f(y):\rr\mapsto\rr$, 
$g(t,y):[0,\infty)\times\rr\mapsto\rr$, $h(y,z):\rr\times G\mapsto\rr$,
we will always denote by $f^{(l)}$,  $g^{(l)}$, and $h^{(l)}$
the $l$-th derivative of $f$, $g$, $h$ with respect to the variable $y$.

\vip

A map $(t,y) \mapsto f(t,y)$ is of 
class $C^{1,k}_b([0,T] \times \rr)$ if 
the derivatives $f^{(l)}(t,y)$ 
and $\partial_t f^{(l)}(t,y)$ exist, are continuous and bounded,  
for all $l\in \{0,...,k\}$ .
\vip

We consider for $i\geq 1$ the approximation
$\cL^i$ of $\cL$, recall (\ref{gi}), defined for all bounded
and measurable $\phi:\rr\mapsto \rr$ by
\begin{equation*}
\cL^i\phi(y)= i
\left[\phi\left(y+ \frac{b(y)}{i}\right) - \phi(y) \right]
+ \gamma(y) \int_{G_i} q(dz) \left[\phi(y+h[y,z])-\phi(y) \right].
\end{equation*}
Here,  $(G_i)_{i\geq 1}$ is an increasing sequence of subsets of $G$ such that 
$\cup_{i\geq 1} G_i = G$ and such that for each $i\geq 1$, 
$q(G_i)<\infty$.

\begin{lem}\label{cestparti}
Assume $(I)$ and $(A_{1,1})$.

(i) For any $i\geq 1$, any probability measure $f_i(dy)$ on $\rr$, there exists
a unique family of (possibly signed) bounded measures 
$(f_i(t,dy))_{t\geq 0}$ on $\rr$ such that for all $T>0$,
$\sup_{[0,T]} \int_{\rr} |f_i(t)|(dy)<\infty$, and
for all bounded measurable $\phi: \rr\mapsto \rr$,
\begin{equation}\label{eql}
\intrd \phi(y)f_i(t,dy)=\intrd \phi(y)f_i(dy)
+ \intot ds \intrd \cL^i \phi(y) f_i(s,dy).
\end{equation}
Furthermore, $f_i(t)$ is a probability measure for all $t\geq 0$. 

(ii) Assume now that $f_i(dy)$ goes weakly  to some probability measure 
$f(dy)$ as $i$ tends to infinity. Then for all $t\geq 0$, 
$f_i(t,dy)$ tends weakly to $p(t,f,dy)$ as $i$ tends to infinity,
where we use the notation of Proposition \ref{propa}.
\end{lem}

\begin{proof}
Let us first prove the uniqueness part.
We observe that
for $\phi$ bounded and measurable, $\cL^i\phi$ is also measurable and
satisfies $|| \cL^i\phi||_\infty \leq C_i ||\phi||_\infty$,
where $C_i := 2i+2||\gamma||_\infty q(G_i)$. Hence for two solutions
$f_i(t,dy)$ and $\tf_i(t,dy)$ to (\ref{eql}), 
an immediate computation leads us to
\begin{equation*}
||f_i(t)-\tf_i(t)||_{TV} \leq C_i \intot ds ||f_i(s)-\tf_i(s)||_{TV},
\end{equation*}
since the total variation norm
satisfies $||\nu||_{TV}:= \sup_{||\phi||_\infty\leq 1} 
|\int_{\rr} \phi(y) \nu(dy)|$. The uniqueness of the solution 
to (\ref{eql}) follows
from the Gronwall Lemma.

Let us consider $X_0\sim f$ independent of $N$, and 
$(X^{x}_t)_{t \geq 0,x\in\rr}$ 
the solution to (\ref{sde}), associated to the
Poisson measure $N$. Recall that $p(t,f,dy)=\cL(X^{X_0}_t)$.

We introduce another Poisson measure
$M^i(ds)$ on $[0,\infty)$ with intensity measure
$i ds$, independent of $N$, and $X^i_0\sim f_i$, independent of $(M^i,N)$. 
Let $(X^{i}_t)_{t\geq 0}$ be the (clearly unique) solution to
\begin{equation*}
X^{i}_t =X_0^i +\intot \frac{b(X^{i}_\sm)}{i}M^i(ds)
+\intot \!\! \int_0^\infty\!\! \int_{G_i}\!\!\!\!
h(X_\sm^{i},z)\indiq_{\{u\leq \gamma(X_\sm^{i})\}}N(ds,du,dz).
\end{equation*}
Then one immediately checks that $f_i(t,dy)=\cL(X^i_t)$ solves (\ref{eql}).
This shows the existence of a solution to (\ref{eql}), and that
this solution consists of a family of probability measures.
Finally, we use the Skorokhod representation Theorem:
we build $X^i_0\sim f_i$ in such a way that $X^i_0$ tends a.s. to $X_0$.
Then one easily proves that $\sup_{[0,t]} |X^i_s-X^{X_0}_s|$ tends 
to $0$ in probability, for all $t\geq 0$,
using repeatedly $(A_{1,1})$. We refer to \cite[Step 1 page 653]{fg} 
for a similar proof. This of course implies that for all $t\geq 0$,
$f_i(t,dy)=\cL(X^i_t)$ tends weakly to $p(t,f,dy)=\cL(X^{X_0}_t)$.
\end{proof}

\vip

We now introduce some inverse functions in order to write
(\ref{eql}) in a strong form.

\begin{lem}\label{inv}
Assume $(S)$ and $(A_{k+1,p})$ for some 
$p\geq k+1 \geq 2$.

(i) For each fixed $z \in G$, the map $y \mapsto y+h(y,z)$
is an increasing $C^{k+1}$-diffeomorphism from $\rr$ into itself.
We thus may introduce its inverse function $\tau(y,z):
\rr \times G \mapsto \rr$ defined by $\tau(y,z)+h(\tau(y,z),z)=y$.
For each $z\in G$, $y\mapsto \tau(y,z)$ is of class $C^{k+1}(\rr)$. 
There exist $\alpha$ and $K>0$ such that
\begin{eqnarray}
&|\tau(y,z)-y| + |\tau'(y,z)-1|+  
\frac{|\tau'(y,z)-1|}{\tau'(y,z)}\leq \alpha(z) \in L^1(G,q), \label{cuh1} \\
&0<\tau'(y,z)\leq K. \label{cuh2}
\end{eqnarray}
For all $l \in \{0,...,k\}$, there exist some functions 
$\alpha_{l,r}:\rr\times G \mapsto \rr$ with 
\begin{eqnarray}\label{cuh3}
\left(1+\frac{1}{\tau'(y,z)}\right) \sum_{r=0}^{l} |\alpha_{l,r}(y,z)|
\leq \alpha(z) \in L^1(G,q)
\end{eqnarray}
such that for all $\phi\in C^l(\rr)$, 
\begin{eqnarray}
\left[\phi(\tau(y,z))\tau'(y,z)\right]^{(l)}
= \phi^{(l)}(\tau(y,z)) + \sum_{r=0}^{l} 
\alpha_{l,r}(y,z) \phi^{(r)}(\tau(y,z)). \label{cuh4}
\end{eqnarray}
(ii) For all $i \geq i_0:= 2|| b'||_\infty$, the map  
$y \mapsto y+b(y)/i$
is an increasing $C^{k+1}$-diffeomorphism from $\rr$ into itself.
Let its inverse $\tau_i: \rr \mapsto \rr$
be defined by $\tau_i(y)+b(\tau_i(y))/i=y$.
Then $\tau_i \in C^{k+1}(\rr)$.
There exists $c>0,K>0$ such that
\begin{equation}\label{cui1}
|\tau_i(y)-y| \leq K/i ,\quad |\tau'_i(y)-1| \leq K/i,
\quad c<\tau'_i(y)\leq K . 
\end{equation}
For all $l \in \{0,...,k\}$, there exist
$\beta^i_{l,r}:\rr \mapsto \rr$ with 
\begin{equation}\label{cui2}
\sum_{r=0}^{l} i |\beta_{l,r}^i(y)| \leq K
\end{equation} 
such that for all $\phi\in C^l(\rr)$, 
\begin{eqnarray}
\left[\phi(\tau_i(y)) \tau'_i(y)\right]^{(l)}
= \phi^{(l)}(\tau_i(y)) 
+ \sum_{r=0}^{l} \beta^i_{l,r}(y) \phi^{(r)}(\tau_i(y)). \label{cui3}
\end{eqnarray}
(iii) For all $i\geq i_0$, all bounded measurable $\phi:\rr\mapsto \rr$ 
and all $g \in L^1(\rr)$, 
$\int_\rr g(y) \cL^i \phi(y) dy  = \int_\rr \phi(y) \cL^{i*} g(y)dy$,
where
\begin{eqnarray}\label{listar}
\cL^{i*} g(y)&=& i \big[g(\tau_i(y))\tau'_i(y) - g(y) \big] \ala
&&+ \int_{G_i} q(dz)
 \big[\gamma(\tau(y,z)) g(\tau(y,z)) \tau'(y,z) 
- \gamma(y) g(y)   \big].
\end{eqnarray}
\end{lem}

\begin{proof} We start with

{\bf Point (i).} The fact that for each $z\in G$, $y+h(y,z)$
is an increasing $C^{k+1}$-diffeomorphism follows immediately
from $(A_{k+1,p})$ and $(S)$. Thus its inverse function 
$y\mapsto \tau(y,z)$ is of class $C^{k+1}$.
Next, $\tau'(y,z)=1/(1+h'(\tau(y,z),z))$, and thus is positive and
bounded by $1/c_0$ due to $(S)$. This shows (\ref{cuh2}).
Of course, $\sup_y |\tau(y,z)-y|=\sup_y |y+h(y,z)-y|\leq\eta(z)\in L^1(G,q)$
due to $(A_{k+1,p})$. Next, 
$|\tau'(y,z)-1|= |h'(\tau(y,z),z)|/(1+h'(\tau(y,z),z)) \leq 
\eta(z)/c_0 \in L^1(G,q)$, due to $(S)$ and $(A_{k+1,p})$. Finally, 
$ |\tau'(y,z)-1|/\tau'(y,z)=|h'(\tau(y,z),z)| \leq 
\eta(z) \in L^1(G,q)$, due to $(A_{k+1,p})$. Thus (\ref{cuh1}) holds.

\vip

We next show that for $l=1,...,k+1$, 
\begin{equation}\label{btl}
|\tau^{(l)}(y,z)|\leq K (\eta(z) +\eta^{l-1}(z)).
\end{equation}
When $l=1$, it suffices to use that $|\tau'(y,z)-1|\leq K\eta(z)$,
which was already proved.
For $l\geq 2$, we use (\ref{dninv}) (with
$f(y)=y+h(y,z)$), the fact that $f'(y)=1+h'(y,z)\geq c_0$ due to $(S)$,
and that for all $n=2,...,k+1$, $f^{(n)}(y)=h^{(n)}(y,z)\leq \eta(z)$
(due to $(A_{k+1,p})$): this yields, 
setting $I_{l,r}:=
\{q\in \nn, \;i_1,...,i_q\in\{2,...,l\};\; i_1+...+i_q=r-1\}$,
\begin{eqnarray*}
|\tau^{(l)}(y,z)|&\leq& K \sum_{r=l+1}^{2l-1} \sum_{I_{l,r}} 
\prod_{j=1}^q |h^{(i_j)}(\tau(y,z),z)| \leq K \sum_{r=l+1}^{2l-1} 
\sum_{I_{l,r}} \eta^q(z) \ala
&\leq& K \sum_{q=1}^{l-1} \eta^q(z) \leq K(\eta(z)+ \eta^{l-1}(z)).
\end{eqnarray*}
\vip

We now consider $\phi\in C^k(\rr)$. Due to (\ref{fdb}),
for $n=1,...,k$,
\begin{equation}\label{rajou}
[\phi(\tau(y,z))]^{(n)}
=[\tau'(y,z)]^n \phi^{(n)} (\tau(y,z)) +
\sum_{r=1}^{n-1} \delta_{n,r}(y,z) \phi^{(r)} (\tau(y,z))
\end{equation}
with $\delta_{n,r}(y,z) = \sum_{J_{n,r}} a^{n}_{i_1,...,i_r} 
\prod_{1}^r \tau^{(i_j)} (y,z)$, where
$J_{n,r}:=\{i_1\geq 1,...,i_r\geq 1, \; i_1+...+i_r=n\}$. 
Using (\ref{btl}), we get, for $r=1,...,n-1$,
\begin{eqnarray}\label{bga}
|\delta_{n,r}(y,z)| &\leq& K \sum_{J_{n,r}} 
\prod_1^r (\eta(z)+\eta^{i_j-1}(z)) 
\leq K \sum_{m=1}^{n-1} \eta^m(z) \ala
&\leq& K (\eta(z) + \eta^{n-1}(z)).
\end{eqnarray}
To obtain the second inequality, we used that since $i_1+...+i_r=n>r$,
there is at least one $j$ with $i_j\geq 2$, and that 
$\sum_{j=1}^r (i_j-1)\lor 1 = \sum_{j=1}^r (i_j-1)
+\sum_{j=1}^r \indiq_{\{i_j=1\}}\leq n-r+r-1=n-1$.

\vip

Applying now the Leibniz formula and then (\ref{rajou}), 
we get, for $l=0,...,k$,
\begin{eqnarray*}
[\phi(\tau)\tau']^{(l)}
 = \tau' [\phi(\tau)]^{(l)} 
+\sum_{n=0}^{l-1} \cnl \tau^{(l+1-n)} [\phi(\tau)]^{(n)} \hskip3cm\ala
= (\tau')^{l+1} \phi^{(l)}(\tau) +
\sum_{r=0}^{l-1} \phi^{(r)}(\tau) \alpha_{l,r}
 = \phi^{(l)}(\tau) +
\sum_{r=0}^{l} \phi^{(r)}(\tau) \alpha_{l,r},
\end{eqnarray*} 
where $\alpha_{l,0}=\tau^{(l+1)}$, $\alpha_{l,l} =  (\tau')^{l+1} - 1$,
and for $r=1,...,l-1$, 
\begin{eqnarray*}
\alpha_{l,r}&=& \crl \tau^{(l+1-r)} (\tau')^{r}
+ \sum_{j=r+1}^{l} \cjl \tau^{(l+1-j)} \delta_{j,r}.
\end{eqnarray*}
It only remains to prove (\ref{cuh3}). First, since $\tau'$ is bounded,
we deduce that $|\alpha_{l,l}(y,z)|\leq K |\tau'(y,z)-1|\leq K \eta(z)$.
Next, using (\ref{btl}), (\ref{bga}) and that $\tau'$ is bounded, 
we get, for $l=1,...,k$, (with the convention $\sum_1^0=0$),
\begin{eqnarray*}
\sum_{r=0}^l |\alpha_{l,r}(y,z)| &\leq& 
K\eta(z) + K(\eta(z)+\eta^l(z))+ K\sum_{r=1}^{l-1} (\eta(z)+\eta^{l-r}(z)) \ala
&& + K\sum_{r=1}^{l-1} \sum_{j=r+1}^l
(\eta(z)+\eta^{l-j}(z))(\eta(z)+\eta^{j-1}(z))\ala
&\leq& K (\eta(z)+\eta^l(z))\leq K (\eta(z)+\eta^k(z))
\end{eqnarray*}
Finally, $(1+1/\tau'(y,z)) = (1+ 1+h'(\tau(y,z),z))
\leq 2+\eta(z)$ by $(A_{k+1,p})$.
We conclude that for $l=1,...,k$,
\begin{equation*}
\left( 1+\frac{1}{\tau'(y,z)}\right) 
\sum_{r=0}^l |\alpha_{l,r}(y,z)| \leq 
K(1+\eta(z))(\eta(z)+\eta^k(z)) =: \alpha (z),
\end{equation*}
and $\alpha \in L^1(G,q)$, 
since by assumption, $\eta \in L^1\cap L^p(G,q)$ with $p\geq k+1\geq 2$.

\vip

{\bf Point (ii).} The proof is the similar (but simpler) to that of Point (i).
We observe that for $i\geq i_0$, 
$(y+b(y)/i)'\geq 1/2$,
so that under $(A_{k+1,p})$, $y+b(y)/i$
is clearly a $C^{k+1}$-diffeomorphism.
Next, (\ref{cui1}) is easily obtained, and we prove as in Point (i)
that
\begin{equation*}
|\tau_i^{(l)}(z)|\leq K(1/i+ (1/i)^{l-1}) \leq K/i ,
\quad l=2,...,k+1,
\end{equation*} 
using that for all $n=2,...,k+1$, $(y+b(y)/i)^{(n)}\leq K/i$
thanks to $(A_{k+1,p})$. Then (\ref{cui2})-(\ref{cui3}) are
obtained as (\ref{cuh3})-(\ref{cuh4}).

\vip

{\bf Point (iii).} Let thus $\phi$ and $g$ as in the statement.
Then
\begin{eqnarray*}
&&\intrd g(y)\cL^i\phi(y)dy = 
i \intrd \phi\left(y+b(y)/i\right) g(y)dy 
- i \intrd \phi\left(y\right) g(y)dy \\
&&\hskip0.5cm +  \int_{G_i} q(dz) \intrd \gamma(y)\phi(y+h[y,z])g(y)dy
-\int_{G_i} q(dz) \intrd \gamma(y)\phi(y)g(y)dy\\
&&\hskip0.5cm =i \intrd \phi\left(y\right) g(\tau_i(y))\tau'_i(y)dy 
- i \intrd \phi\left(y\right) g(y)dy \\
&&\hskip0.5cm +  \int_{G_i} q(dz) \intrd \gamma(\tau(y,z))\phi(y)g(\tau(y,z))
\tau'(y,z)dy \\
&&\hskip0.5cm -\int_{G_i} q(dz) \intrd \gamma(y)\phi(y)g(y)dy
=\intrd \phi(y)\cL^{i*}g(y)dy,
\end{eqnarray*}
where we used the substitution $y \mapsto \tau_i(y)$ 
(resp. $y\mapsto\tau(y,z)$)
in the first (resp. third) integral.
\end{proof}

The following technical lemma shows that 
when starting with a smooth initial condition, the solution of (\ref{eql})
remains smooth for all times (not uniformly in $i$).
This will enable us to handle rigorous computations.

\begin{lem}\label{cestparti2}
Assume $(I)$, $(A_{k+1,p})$ for some $p\geq k+1\geq 2$, and $(S)$.
Let $i \geq i_0$ be fixed.
Consider a probability measure $f_i(dy)$ admitting 
a density $f_i(y)$ of class $C^k_b(\rr)$, and the associated
solution $f_i(t,dy)$ to (\ref{eql}).
Then for all $t\geq 0$, $f_i(t,dy)$ has a density $f_i(t,y)$,
and $(t,y) \mapsto f_i(t,y)$ belongs to $C^{1,k}_b([0,T] \times \rr)$ for all $T\geq 0$.
For all $t\geq 0$,
all $y \in \rr$, all $l=0,...,k$,
\begin{eqnarray}\label{dereql}
\partial_{t} f_i^{(l)}(t,y) &=& \big[\cL^{i*}f_i(t,y) \big]^{(l)} \ala
&=&
i \big[f_i(t,\tau_i(y))\tau'_i(y) - f_i(t,y) \big]^{(l)} \\
&&+ \int_{G_i} q(dz)
 \big[\gamma(\tau(y,z)) f_i(t,\tau(y,z)) \tau'(y,z) 
- \gamma(y) f_i(t,y)   \big]^{(l)}.\nonumber
\end{eqnarray}
\end{lem}

\begin{proof}
We will prove, using a Picard iteration,
that (\ref{dereql}) (with $l=0$) admits a solution, which also solves
(\ref{eql}), which is regular, and of which the derivatives solve
(\ref{dereql}). We omit the fixed subscript $i\geq i_0$ in this part
of the proof, and the initial probability measure 
$f(dy)=f(y)dy$ with $f \in C^k(\rr)$ is fixed. 

\vip

{\bf Step 1.} 
Consider the function $f^0(t,y):=f(y)$, and define, for $n\geq 0$,
\begin{eqnarray}\label{picpic}
f^{n+1}(t,y) = f(y)+ \intot \cL^{i*}f^n(s,y) ds.
\end{eqnarray}
Then one easily checks by induction (on $n$), 
using Lemma \ref{inv}, $(A_{k+1,p})$ and the fact that $q(G_i)<\infty$, 
that for all $n\geq 0$,
$f^n(t,y)$ is of class $C^{0,k}_b([0,\infty)\times\rr)$,
and that for all $l\in \{0,...,k\}$,
\begin{eqnarray}\label{picpicl}
&& (f^{n+1})^{(l)}(t,y) =  f^{(l)}(y) + \intot [\cL^{i*}f^n]^{(l)}(s,y) ds.
\end{eqnarray}

{\bf Step 2.} We now show that there exists $C_{k,i}>0$
such that for $n\geq 1$, $t\geq 0$,
\begin{eqnarray*}
\sum_{l=0}^k ||(\delta^{n+1})^{(l)}(t,.) ||_{\infty} 
\leq C_{k,i} \intot ds \sum_{l=0}^k 
|| (\delta^{n})^{(l)}(s,.) ||_{\infty},
\end{eqnarray*}
where $\delta^{n+1}(t,y)=f^{n+1}(t,y)-f^n(t,y)$.
Due to (\ref{picpicl}), for $l=0,...,k$,
\begin{eqnarray*}
(\delta^{n+1})^{(l)}(t,y) = 
 \intot i \big[\delta^{n}(s,\tau_i(y))\tau'_i(y) 
- \delta^{n}(s,y) \big]^{(l)} ds \hskip2cm\ala
+ \intot ds \int_{G_i} q(dz)
\big[\gamma(\tau(y,z)) \delta^n(s,\tau(y,z))\tau'(y,z) 
- \gamma(y) \delta^{n}(s,y)   \big]^{(l)}.
\end{eqnarray*}
We now use (\ref{cui3}) (with $\phi=\delta^n(s,.)$) and (\ref{cuh4}) 
(with $\phi=\gamma\delta^n(s,.)$),
and we easily obtain, since $q(G_i)<\infty$,
for some constant $C_{k,i}$, for all $y\in\rr$,
\begin{eqnarray*} 
|(\delta^{n+1})^{(l)}(t,y)| &\leq& 
C_{k,i} \intot \! ds  \sum_{r=0}^{l} \left(
||(\delta^n)^{(r)}(s) ||_\infty  + ||(\gamma\delta^n)^{(r)}(s) ||_\infty \right)\\
&\leq& 
C_{k,i} \intot \!  ds  \sum_{r=0}^{l}
||(\delta^n)^{(r)}(s) ||_\infty,
\end{eqnarray*}
the last inequality holding since $l\leq k$ and $\gamma \in C^k_b(\rr)$.
Taking now the supremum over $y\in\rr$ and suming for $l=0,...,k$, we get
the desired inequality.

\vip

{\bf Step 3.} We classically 
deduce from Step 2 that the sequence $f^n$ tends to a function
$f(t,y)\in C^{0,k}_b([0,T]\times\rr)$ (for all $T>0$),
and that for $l=0,...,k$,
\begin{equation}\label{inteql}
f^{(l)}(t,y)=f^{(l)}(y) + \intot [\cL^{i*} f]^{(l)}(s,y) ds.
\end{equation}
But one can check, using arguments as in Step 1, that
since $f(t,y)\in C^{0,k}_b([0,T]\times\rr)$, so does
$[\cL^{i*} f](t,y)$. Hence (\ref{inteql}) can be differentiated
with respect to time, we obtain (\ref{dereql}), and thus also
that $f(t,y)\in C^{1,k}_b([0,T]\times\rr)$.

\vip

{\bf Step 4.} It only remains to show that $f(t,y)dy$ is indeed the solution
of (\ref{eql}) defined in Lemma \ref{cestparti}-(i).
First, using (\ref{picpic}) and rough estimates,
we have $||f^{n+1}(t)||_{L^1} \leq ||f||_{L^1} + C_i \int_0^t ds  ||f^{n}(s)||_{L^1}$,
where $C_i=2i + 2||\gamma||_\infty q(G_i)$.
This classically ensures that  $||f(t)||_{L^1} \leq \limsup_n 
||f^{n}(t)||_{L^1}
\leq  ||f||_{L^1}e^{C_i t}$.
Thus $\sup_{[0,T]} \int_\rr |f(t,y)|dy<\infty$ for all $T>0$.

Next, we multiply (\ref{inteql}) (with $l=0$) by $\phi(y)$,
for a bounded measurable $\phi:\rr\mapsto\rr$, we integrate
over $y\in \rr$, and we use the duality proved in 
Lemma \ref{inv}-(iii). This yields (\ref{eql}).
\end{proof}

\vip

The central part of this section consists of the following
result.

\begin{lem}\label{mmp}
Assume $(I)$, $(S)$ and $(A_{k+1,p})$ for some
$p\geq k+1\geq 2$. For $i\geq i_0$, 
let $f_i(dy)\in \wbk$ be a probability measure with a 
density $f_i(y)\in C^k(\rr)$, and consider the unique
solution $f_i(t,dy)$ to (\ref{eql}).
There exists a constant $C_k$ (not depending on $i\geq i_0$)
such that for all $t\geq 0$,
\begin{equation*}
||f_i(t,.)||_{\wbk} \leq ||f_i||_{\wbk} e^{C_k t}.
\end{equation*}
\end{lem}

\begin{proof}
We know from Lemma \ref{cestparti2} that $f_i(t,y)$ is of class
$C^{1,k}_b([0,T] \times \rr)$, and that
(\ref{dereql}) holds for $l=0,...,k$. 

Since for each $l=0,...,k$, each $y\in \rr$,
$t\mapsto f_i^{(l)}(t,y)$ is of class $C^1$, we classically deduce that 
$|f_i^{(l)}(t,y)|= |f_i^{(l)}(y)|+\int_0^t  sg(f_i^{(l)}(s,y)) 
\partial_t f_i^{(l)}(s,y) ds$, where
$sg(u)=\indiq_{(0,\infty)}(u)-\indiq_{(-\infty,0)}(u)$. Using thus (\ref{dereql})
and integrating over $y\in\rr$, we get
\begin{equation}
||f_i^{(l)}(t,.)||_{L^1}= ||f_i^{(l)}||_{L^1} + \intot (A_i^l(s)+B_i^l(s))ds,
\end{equation}
for $l=1,...,k$, where, setting $\gamma f_i(t,y)= \gamma(y) f_i(t,y)$ 
for simplicity,
\begin{eqnarray*}
A_i^l(t)&=&\intrd dy \; i \big[f_i(t,\tau_i(y))\tau'_i(y) 
- f_i(t,y) \big]^{(l)} sg(f_i^{(l)}(t,y))\\
B_i^l(t)&=& \int_{G_i}q(dz) \intrd dy \; i \big[\gamma f_i(t,\tau(y,z))
\tau'(y,z) 
- \gamma f_i(t,y) \big]^{(l)} sg(f_i^{(l)}(t,y)).\nonumber
\end{eqnarray*}
Using (\ref{cui3}) (with $\phi=f_i(t,.)$) and then (\ref{cui2}), we obtain
\begin{eqnarray*}
A_i^l(t)&\leq &\intrd dy \; i \big[ f_i^{(l)}(t,\tau_i(y))
-f_i^{(l)}(t,y)) \big] sg(f_i^{(l)}(t,y))\\
&&+ \intrd dy \; \sum_{r=0}^l i |\beta^i_{l,r}(y)|.|f_i^{(r)}(t,\tau_i(y))|
\ala
&\leq & \intrd dy \; i \big[ |f_i^{(l)}(t,\tau_i(y))|
- |f_i^{(l)}(t,y)| \big] + K \intrd dy \; \sum_{r=0}^l |f_i^{(r)}(t,\tau_i(y))|\\
&=:& A_i^{l,1}(t)+A_i^{l,2}(t).
\end{eqnarray*}
First, 
\begin{eqnarray*}
A_i^{l,1}(t) 
&\leq &i \intrd dy |f_i^{(l)}(t,\tau_i(y))| \tau_i'(y)
- i \intrd dy |f_i^{(l)}(t,y)|\\
&& + \intrd dy |f_i^{(l)}(t,\tau_i(y))| \times i|\tau_i'(y)-1|.
\end{eqnarray*}
Using the substitution $\tau_i(y) \mapsto y$ in the first integral,
we deduce that the first and second integral are equal. Next, due to
(\ref{cui1}), we get
\begin{eqnarray*}
A_i^{l,1}(t) &\leq&  0 +  K \intrd dy |f_i^{(l)}(t,\tau_i(y))|\leq K
||f_i^{(l)}(t,.)||_{L^1}.
\end{eqnarray*}
To obtain the last inequality, we used again the substitution
$\tau_i(y) \mapsto y$ and the fact that $\tau_i'$ is bounded below
(uniformly in $i\geq i_0$, see (\ref{cui1})). 
The same argument shows that
\begin{eqnarray*}
A_i^{l,2}(t) &\leq&  K \sum_{r=0}^l ||f_i^{(r)}(t,.)||_{L^1}.
\end{eqnarray*}
Using now (\ref{cuh4}) with $\phi=\gamma f_i(t,.)$, we get
\begin{eqnarray*}
B_i^l(t)&\leq& \intgi \!\!q(dz)\!\!\intrd dy  \big[ (\gamma
f_i)^{(l)}(t,\tau(y,z)) - (\gamma f_i)^{(l)}(t,y)\big] sg(f_i^{(l)}(t,y))\\
&&+\intgi\!\! q(dz)\!\! \intrd dy 
\sum_{r=0}^l  |\alpha_{l,r}(y,z)|.|(\gamma f_i(t,.))^{(r)}(\tau(y,z))
\end{eqnarray*}
With the help of the Leibniz formula, we obtain
\begin{eqnarray*}
B_i^l(t)
&\leq& \intgi \!\!q(dz)\!\!\intrd dy  \big[ \gamma f_i^{(l)}(t,\tau(y,z)) -
\gamma f_i^{(l)}(t,y)\big] sg(f_i^{(l)}(t,y))\\
&&+ \intgi \!\!q(dz)\!\!\intrd dy  \sum_{r=0}^{l-1} \crl 
\big| \gamma^{(l-r)} f_i^{(r)}(t,\tau(y,z))-\gamma^{(l-r)} 
f_i^{(r)}(t,y) \big|\\
&&+\intgi\!\! q(dz)\!\! \intrd dy \sum_{r=0}^l  
|\alpha_{l,r}(y,z)|.|(\gamma f_i(t,.))^{(r)}(\tau(y,z))| \ala
&=:& B_i^{l,1}(t)+B_i^{l,2}(t)+B_i^{l,3}(t).\hskip6cm
\end{eqnarray*}
First, 
\begin{eqnarray*}
B_i^{l,1}(t)\leq \intgi \!\!q(dz)\!\!\intrd dy  \big[ (\gamma 
|f_i^{(l)}|)(t,\tau(y,z)) . \tau'(y,z) -(\gamma |f_i^{(l)}|) (t,y)|\big]\\
+ \intgi \!\!q(dz)\!\!\intrd dy   (\gamma
|f_i^{(l)}|)(t,\tau(y,z)) \times | \tau'(y,z)-1|.
\end{eqnarray*}
Using the substitution $\tau(y,z) \mapsto y$ is the first part of the first 
integral, we deduce that the first integral equals $0$.
Since $\gamma$ is bounded, we get
\begin{eqnarray*}
B_i^{l,1}(t)&\leq& 0 +
K \intgi \!\!q(dz)\!\!\intrd dy |f_i^{(l)}(t,\tau(y,z))| \times | \tau'(y,z)-1|\\
&\leq&  K \intgi \!\!\alpha(z)q(dz)\!\!\intrd dy |f_i^{(l)}(t,\tau(y,z))| \tau'(y,z)
\end{eqnarray*}
for some $\alpha \in L^1(G,q)$, where we used (\ref{cuh1}).
But using again the subsitution $\tau(y,z) \mapsto y$, we find
\begin{eqnarray*}
B_i^{l,1}(t)\leq K \intgi \alpha(z) q(dz)\!\!\intrd dy 
|f_i^{(l)}(t,y)| \leq K ||f_i^{(l)}(t,.)||_{L^1}.
\end{eqnarray*}
Next, using (\ref{cuh3}), then the substitution 
$\tau(y,z) \mapsto y$ and that $\gamma \in C^k_b(\rr)$, we obtain, 
for some $\alpha \in L^1(G,q)$,
\begin{eqnarray*}
B_i^{l,3}(t)\leq K \sum_{r=0}^l \intgi \!\!q(dz)\!\!\intrd dy  
|(\gamma f_i(t,.))^{(r)}(\tau(y,z))| \tau'(y,z) \alpha(z) \\
\leq  \sum_{r=0}^l \left(\intgi \!\! \alpha(z) q(dz) \right)
||(\gamma f_i(t,.)^{(r)}||_{L^1}
\leq K \sum_{r=0}^l || f_i^{(r)}(t,.)||_{L^1}
\end{eqnarray*}
Finally, due to (\ref{cuh1}), there exists $\alpha \in L^1(G,q)$
such that $\sup_y |\tau(y,z)-y|\leq\alpha(z)$. Hence for any
$\phi \in C^1(\rr)$,
\begin{eqnarray*}
\intgi q(dz) \intrd dy |\phi(\tau(y,z))-\phi(y)| \leq 
\intgi q(dz) \intrd dy  \int_{y-\alpha(z)}^{y+\alpha(z)} du |\phi'(u)|\ala
\leq 2 \intgi \alpha(z)q(dz) ||\phi'||_{L^1} \leq K ||\phi'||_{L^1}.
\end{eqnarray*}
As a consequence, using that $\gamma \in C^{k+1}_b$, we get, since $l\leq k$,
\begin{equation*}
B_i^{l,2}(t) \leq  K \sum_{r=0}^{l-1} 
|| (\gamma^{(l-r)}f_i^{(r)})'(t,.)||_{L^1}
\leq K \sum_{r=0}^{l} || f_i^{(r)}(t,.)||_{L^1}.
\end{equation*}
We finally have proved that for $l=1,...,k$, for all $t\geq 0$,
$$
||f_i^{(l)}(t,.)||_{L^1} \leq ||f_i^{(l)}||_{L^1}+ K \sum_{r=0}^l
\intot ds ||f_i^{(r)}(s,.)||_{L^1}.
$$
Using that for all $t\geq 0$, $f_i(t,.)$ is a probability measure
(so that $||f_i(t,.)||_{L^1}=1$) and
summing over $l=0,...,k$, we immediately conclude that 
$$
||f_i(t,.)||_{\wbk} \leq ||f_i||_{\wbk}+ K \intot ds ||f_i(s,.)||_{\wbk}.
$$
The Gronwall Lemma allows us to conclude the proof.
\end{proof}

We finally conclude the

\vip

\begin{proof} {\bf of Proposition \ref{propa}.}
We thus assume $(I)$, $(A_{k+1,p})$ for some $p\geq k+1 \geq 2$, and $(S)$.
Consider a probability measure $f\in \wbk$, and a sequence of
probability measures $f_i\in \wbk$ with densities $f_i \in C^k_b(\rr)$,
such that $f_i$ goes weakly to $f$, and such that
$\lim_i ||f_i||_{\wbk} = ||f||_{\wbk}$.
Consider the unique solution $f_i(t,y)$ to (\ref{eql}).
Then we deduce from Proposition \ref{mmp} that for $t\geq 0$,
\begin{equation}\label{jab1}
||f_i(t,.)||_{\wbk} \leq ||f_i||_{\wbk} e^{C_k t}.
\end{equation}
On the other hand, Lemma \ref{cestparti} implies that for all $t \geq 0$,
$f_i(t,dy)$ goes weakly to $p(t,f,dy)$ as $i$ tends to infinity.
Thus for any $ \phi \in C^k_b(\rr)$, any $l\in \{0,...,k\}$, any $t\geq 0$,
\begin{equation*}
\int_\rr \phi^{(l)} (y)p(t,f,dy) = \lim_{i\to\infty}
\int_\rr \phi^{(l)}(y) f_i(t,dy).
\end{equation*}
We then immediately deduce from (\ref{jab1}), recalling (\ref{dfn}), that
for any $t\geq 0$,
\begin{equation*}
||p(t,f,.)||_{\wbk} \leq ||f||_{\wbk}e^{C_k t}.
\end{equation*}
The proof is finished.
\end{proof}

\section{Appendix}

We first gather some formulae about derivatives of composed and inverse
functions from $\rr$ into itself. Here $f^{(l)}$ stands for the $l$-th
derivative of $f$.

\vip

Let us recall the Faa di Bruno formula. Let $l\geq 1$ be fixed.
The exist some coefficients $a^{l,r}_{i_1,...,i_r}>0$  such that 
for $\phi:\rr\mapsto\rr$ and $\tau:\rr\mapsto\rr$ of class $C^l(\rr)$,
\begin{eqnarray}\label{fdb}
[\phi(\tau)]^{(l)}= [\tau']^l \phi^{(l)}(\tau)
+ \sum_{r=1}^{l-1} \left(\sum_{i_1+...+i_r=l}
a^{l,r}_{i_1,...,i_r}  \prod_{j=1}^r \tau^{(i_j)} \right) 
\phi^{(r)}(\tau),
\end{eqnarray}
where the sum is taken over $i_1\geq 1$, ..., $i_r\geq 1$
with $i_1+...+i_r=l$.

\vip

We carry on with another formula.
For $l\geq 2$ fixed, there 
exist some coefficients $c^{l,r}_{i_1,..,i_q}\in\rr$ such that
for $f:\rr\mapsto\rr$ a $C^l$-diffeomorphism,
and for $\tau$ its inverse function,
\begin{equation}\label{dninv}
\tau^{(l)}=\sum_{r=l+1}^{2l-1} \frac{1}{(f'(\tau))^r}
\sum_{i_1+...+i_q=r-1} c^{l,r}_{i_1,..,i_q}\prod_{j=1}^q f^{(i_j)}(\tau),
\end{equation}
where the sum is taken over $q\in\nn$, over $i_1,...,i_q\in \{2,...,l\}$
with $i_1+...+i_q=r-1$. This (not optimal) formula can be checked 
by induction on $k\geq 2$.

\vip

We finally give the

\vip

\begin{proof} {\bf of Lemma \ref{tract}.}
In the whole proof, $y\in \rr$ is fixed. We assume for example
that $I(y)=(a(y),\infty)$, and we may suppose without loss
of generality that $a(y)=0$ (replacing 
if necessary $h[y,z]$ by $\bar h(y,z):=h[y,z +a(y)]$).

We introduce a family of $C^\infty$
functions $\phi_n:\rr \mapsto [0,1]$,
such that $\phi_n(z)=0$ for $z\leq 1$ and $z\geq n+3$,
$\phi_n(z)=1$ for $z\in [2,n+2]$, and $\sup_n ||\phi_n^{(l)}||_\infty \leq C_l$
for all $l\in \nn$. Then we set $q_n(y,dz)=\phi_n(\gamma(y).z)dz$,
and we define $\mu_n$ by 
$\mu_n(y,A)=\gamma(y) \int_G \indiq_A(h(y,z))q_n(y,dz)$.

Clearly $0 \leq q_n(y,dz) \leq q(dz)$ so that $\mu_n(y,du)\leq \mu(y,du)$, 
and an immediate
computation leads us to $\mu_n(y,\rr)=\gamma(y)q_n(y,G) \in [n,n+2]$.
Thus points (i) and (ii) of assumption $(H_{k,p,\theta})$ are
fulfilled.

Since $h'_z(y,z)$ does never vanish, $z \mapsto h(y,z)$ is either
increasing or decreasing. We assume for example
that we are in the latter case. We also necessarily have 
$\lim_{z\to \infty} h(y,z)=0$, since 
$h(y,z) \in L^1((0,\infty),dz)$ (due to $(A_{1,1})$). As a conclusion,
$z \mapsto h(y,z)$ is a
decreasing $C^{k+1}$-diffeomorphism from $(0,\infty)$ into $(0,h(y,0))$.

Let $\xi(y,.):(0,h(y,0))\mapsto (0,\infty)$ be its inverse, that is 
$h(y,\xi(y,u))=u$. 
Then by definition of $\mu_n$ and by using the subsitution 
$u=h(y,z)$, we get $\mu_n(y,du)=\mu_n(y,u)du$ with
\begin{equation}\label{fni}
\mu_n(y,u)=\gamma(y) \phi_n(\gamma(y)\xi(y,u))
\xi'_u(y,u) \indiq_{\{u \in (0,h(y,0)) \}}.
\end{equation}
Since the properties of $\phi_n$ ensure us that $\mu_n(y,u)=0$
for 
\begin{equation*}
u \notin (h[y,(n+2)/\gamma(y)],h[y,1/\gamma(y)]),
\end{equation*} 
it suffices to study the regularity of $\mu_n(y,.)$ on $(0,h(y,0))$. 
Since $\xi(y,.)$ is of class $C^{k+1}$ and since $\phi_n$ is
$C^\infty$, we deduce that $\mu_n(y,.)$ is $C^k$ on $(0,h(y,0))$
(and thus on $\rr$).

Using (\ref{dninv}) and that $h^{(l)}_z$ is uniformly bounded 
(for all $l=1,...,k+1$), we easily get, for $l=2,...,k+1$,
\begin{eqnarray*}
| \xi^{(l)}_u (y,u) | \leq   K \sum_{r=l+1}^{2l-1} |h'_z(y,\xi(y,u))|^{-r} 
\end{eqnarray*}
Since $h'_z$ is uniformly bounded,
we get
\begin{eqnarray}\label{bxi}
| \xi^{(l)}_u (y,u) | \leq   K 
|h'_z(y,\xi(y,u))|^{-2l+1},
\end{eqnarray}
and the formula holds for $l=1,...,k+1$ (when $l=1$, it is obvious).

Applying now (\ref{fdb}), using (\ref{bxi}) and 
that $\gamma$ is bounded, we get,
for $l=1,...,k$,
\begin{eqnarray*}
|[\phi_n(\gamma(y)\xi(y,u))]^{(l)}_u |&\leq& K \sum_{r=1}^l
|h'_z(y,\xi(y,u))|^{-2l+r} \phi_n^{(r)}(\gamma(y).\xi(y,u)).
\end{eqnarray*}
We used here that for $i_1\geq 1,...,i_r\geq 1$ with $i_1+...+i_r=l$,
one has the inequality $\prod_{j=1}^r |\xi^{(i_j)}_u(y,u)| \leq K 
|h'_z(y,\xi(y,u))|^{\sum_1^r(-2 i_j+1)}\leq  
K |h'_z(y,\xi(y,u))|^{-2l+r}$.
Hence
\begin{eqnarray}\label{rr}
|[\phi_n(\gamma(y)\xi(y,u))]^{(l)}_u |
&\leq& K |h'_z(y,\xi(y,u))|^{-2l+1} \indiq_{\{\gamma(y)\xi(y,u)\leq n+3\}},
\end{eqnarray}
since $h'_z$ is uniformly bounded and $\phi_n(z)=0$
for $z\geq n+3$.

Applying finally the Leibniz formula,
using (\ref{fni}), (\ref{bxi}) and (\ref{rr}), 
we get, for $l=1,...,k$, for $u \in (0,h(y,0))$,
\begin{eqnarray*}
| (\mu_n)^{(l)}_u (y,u) | &\leq&  K \gamma(y) 
\sum_{r=0}^l | \xi^{(l+1-r)}_u(y,u) | \times 
|[\phi_n(\gamma(y)\xi(y,u))]^{(r)}_u | \ala
& \leq & K \gamma(y)  |h'_z(y,\xi(y,u))|^{-2l-1}
\indiq_{\{ \xi(y,u) \leq (n+3)/\gamma(y)\}}
\end{eqnarray*}
and the formula obviously holds for $l=0$. Finally,
since $h'_z$ is uniformly bounded, and performing the substitution
$z=\xi(y,u)$, i.e. $u=h(y,z)$, we obtain, 
recalling that $\mu_n(y,\rr) \in [n,n+2]$,
\begin{eqnarray*}
&&\frac{1}{\mu_n(y,\rr)}||\mu_n(y,.)||_\wbk \ala
&&\hskip1cm \leq \frac{K \gamma(y)}{n}
\intrd \sum_{l=0}^k  |h'_z(y,\xi(y,u))|^{-2l-1}
\indiq_{\{ 0<\xi(y,u) \leq \frac{ (n+3)}{\gamma(y)}\}} du \ala 
&&\hskip1cm\leq \frac{K \gamma(y)}{n}
\intrd |h'_z(y,\xi(y,u))|^{-2k-1}
\indiq_{\{ 0<\xi(y,u) \leq \frac{ (n+3)}{\gamma(y)}\}} du \ala
&&\hskip1cm\leq \frac{K \gamma(y)}{n} \intrd
|h'_z(y,z)|^{-2k} \indiq_{\{ 0<z\leq \frac{ (n+3)}{\gamma(y)}\}} dz
\leq KCe^{3\theta} (1+|y|^p)e^{\theta n},
\end{eqnarray*}
where we finally used (\ref{albr})
(because $I_n(y)=[0,n/\gamma(y)]$ here).
This proves that $(H_{k,p,\theta})$-(iii) holds.
\end{proof}

%%%%%%%%%%%%%%%%%%%%%%%%%%%%%%%%%%%%%%%%%
%%%%%%%%% Biblio %%%%%%%%%%%%%%%%%%%%%%%%
%%%%%%%%%%%%%%%%%%%%%%%%%%%%%%%%%%%%%%%%%

\end{document}